\documentclass[12pt]{article}

\usepackage{graphicx}
\usepackage{amsthm,xcolor}
\usepackage{amsfonts}
\usepackage{amsmath}
\usepackage{amscd}
\usepackage{amssymb}
\usepackage{alltt}
\usepackage{url}
\usepackage{ellipsis}
\usepackage{tikz} %
\usetikzlibrary{chains,shapes,arrows,%
 trees,matrix,positioning,decorations}

\def\tikzfig#1#2#3{%
\begin{figure}[htb]%
  \centering
\begin{tikzpicture}#3
\end{tikzpicture}
  \caption{#2}
  \label{fig:#1}%
\end{figure}%
}
\def\smalldot#1{\draw[fill=black] (#1) %
 node [inner sep=1.3pt,shape=circle,fill=black] {}}

\newtheorem{theorem}{Theorem}[subsection]
\newtheorem{lemma}[theorem]{Lemma}

\newcommand{\ring}[1]{\mathbb{#1}}
\newcommand{\op}[1]{\hbox{#1}}
\newcommand{\f}[1]{\frac{1}{#1}}
\newcommand{\Eaff}{E_{\op{\scriptsize\it aff}}}
\newcommand{\Eaf}[1]{E_{\op{\scriptsize\it aff},{#1}}}
\newcommand{\Eoo}{E^{\circ}} 

\newcommand{\Go}{\langle\rho\rangle}

\newcommand{\ang}[1]{\langle{#1}\rangle}

\title{The Group Law for Edwards Curves}
\author{Thomas C. Hales}
\date{}

\begin{document}

\maketitle

\begin{abstract} 
  This article gives an elementary computational proof of the group
  law for Edwards elliptic curves following Bernstein, Lange, et al.,
  Edwards, and Friedl.  The associative law is expressed as a
  polynomial identity over the integers that is directly checked by
  polynomial division.  No preliminaries such as intersection numbers,
  B\'ezout's theorem, projective geometry, divisors, or Riemann Roch
  are required.  The proofs have been designed to facilitate the
  formal verification of elliptic curve cryptography.
\end{abstract}

\parskip=0.8\baselineskip
\baselineskip=1.05\baselineskip

\newenvironment{blockquote}{%
  \par%
  \medskip%
  \baselineskip=0.7\baselineskip%
  \leftskip=2em\rightskip=2em%
  \noindent\ignorespaces}{%
  \par\medskip}

This article started with my frustration in teaching the elliptic
curve group law in an undergraduate course in cryptography.  I needed
a simple proof of associativity.  At the same time, my work on the
formal verification of mathematics made me wary of the so-called
simple proofs on the internet that achieve their simplicity by
skipping cases or by relying on unjustified machinery.

Edwards curves have been widely promoted because their addition law
avoids exceptional cases.  It is natural to ask whether the proof of
the associative law also avoids exceptional cases when expressed in
terms of Edwards curves.  Indeed, this article gives a two-line proof
of the associative law for ``complete'' Edwards curves that avoids
case splits and all the usual machinery.

At the same time, we motivate the addition law.  The usual chord and
tangent addition law for Weierstrass curves can seem terribly
unmotivated at first sight.  We show that the group law for a circle
can be described by a geometric construction, which motivates elliptic
curve addition, because the same geometric construction applied to an
Edwards curve gives its group law, by \cite{arene2011faster}.  One
pleasant surprise is that our proof of the group axioms applies
uniformly to both the circle and the Edwards curve.

\section{The Circle}

The unit circle $\ring{C}_1$ in the complex plane $\ring{C}$ is a
group under complex multiplication, or equivalently under the addition
of angles in polar coordinates:
\begin{equation}\label{eqn:cx}
(x_1,y_1) * (x_2,y_2) = (x_1 x_2 - y_1 y_2,x_1 y_2 + x_2 y_1).
\end{equation}
We write $\iota(x,y) = (x,-y)$ for complex conjugation, the inverse in
$\ring{C}_1$.

We give an unusual interpretation of the group law on the unit circle
that we call {\it hyperbolic addition}.  We consider the family of
hyperbolas in the plane that pass through the point $z=(-1,0)$ and
whose asymptotes are parallel to the coordinate axes.  The equation of
such a hyperbola has the form
\begin{equation}\label{eqn:hyp}
x y + p (x+1) + q y = 0.
\end{equation}
All hyperbolas in this article are assumed to be of this form.  As
special cases (such as $x y=0$), this family includes pairs of lines.
 
Every two points $z_1$ and $z_2$ on the unit circle intersect some
hyperbola within the family.  This incidence condition uniquely
determines $p$ and $q$ when $(-1,0)$, $z_1$ and $z_2$ are not
collinear.  As illustrated in Figure~\ref{fig:fig1}, the hyperbola
meets the unit circle in one additional point $z_3 = (x_3,y_3)$.  The
following remarkable relationship holds among the three points $z_1$,
$z_2$, and $z_3$ on the intersection of the circle and hyperbola.

\begin{lemma}[hyperbolic addition on the circle] \label{lemma:ha} Let
  $z_0=(-1,0)$, $z_1$, $z_2$ and $z_3$ be four distinct points on the
  intersection of the unit circle with a hyperbola in the family (\ref{eqn:hyp}).  
  Then $z_1 z_2 z_3 = 1$ in $\ring{C}_1$.
\end{lemma}

The lemma is a special case of a more general lemma
(Lemma~\ref{lemma:hyperbola}) that is proved later in this article.

This gives a geometric construction of the group law: the product of
the two points $z_1$ and $z_2$ on the unit circle is $\iota(z_3)$.
Rather than starting with the standard formula for addition in
$\ring{C}_1$, we can reverse the process, defining a binary operation
$(\oplus)$ on the circle by setting
\[
z_1 \oplus z_2 = \iota(z_3)
\]
whenever $z_1$, $z_2$, and $z_3$ are related by the circle and
hyperbola construction.  We call the binary operation $\oplus$ {\it
  hyperbolic addition} on the circle.

It might seem that there is no point in reinterpreting complex
multiplication on the unit circle as hyperbolic addition, because they
are actually the same binary operation, and the group $\ring{C}_1$ is
already perfectly well-understood.  However, in the next section, we
will see that hyperbolic addition generalizes in ways that ordinary
multiplication does not.  In this sense, we have found a better
description of the group law on the circle.  The same description
works for elliptic curves!


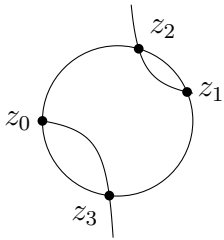
\tikzfig{fig1}{A unit circle centered at the origin
and hyperbola meet at four points $z_0 = (-1,0)$, $z_1$, $z_2$, and $z_3$,
where $z_1 z_2 z_3 = 1$, which we write alternatively in additive notation 
as $z_1\oplus z_2 = \iota(z_3)$.}{
{
\draw (0,0) circle (1);
\draw plot[smooth] file {figC1.table};
\draw plot[smooth] file {figC2.table};
\smalldot {-1,0} node[anchor=east] {$z_0$};
\smalldot {12/13,5/13} node[anchor=west] {$z_1$};
\smalldot {7/25,24/25} node[anchor=south west] {$z_2$};
\smalldot {-36/325,-323/325} node[anchor=north east] {$z_3$};
}
}

\section{Deforming the Circle}

We can use exactly the same hyperbola construction to define a binary
operation $\oplus$ on other curves.  We call this {\it hyperbolic
  addition} on a curve.  We replace the unit circle with a more
general algebraic curve $C$, defined by the zero set of
\begin{equation}\label{eqn:ed}
x^2 + c y^2 - 1 - d x^2 y^2
\end{equation} for some parameters $c$ and $d$.
This zero locus of this polynomial is called an {\it Edwards curve}.%
\footnote{This definition is more inclusive than definitions stated
  elsewhere.  Most writers prefer to restrict to curves of genus one
  and generally call a curve with $c\ne 1$ a {\it twisted Edwards
    curve}.  We have interchanged the $x$ and $y$ coordinates on the
  Edwards curve to make it consistent with the group law on the
  circle.}  The unit circle corresponds to parameter values $c=1$ and
$d=0$.

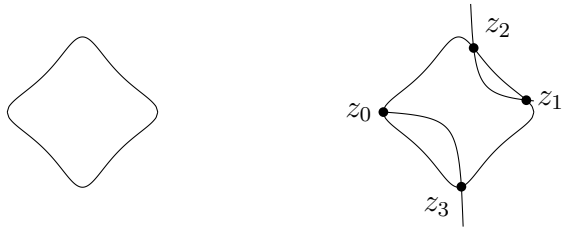
\tikzfig{fig2}{The figure on the left is an Edwards curve (with
  parameters $c=0$ and $d=-8$).  An Edwards curve and hyperbola meet
  at four points $z_0 = (-1,0)$, $z_1$, $z_2$, and $z_3$.  By
  construction, hyperbolic addition satisfies $z_1 \oplus z_2 =
  \iota(z_3)$.}
{
{
\begin{scope}[xshift=0]
\draw plot[smooth] file {figE1.table};
\draw plot[smooth] file {figE2.table};
\end{scope}
\begin{scope}[xshift=5cm]
\draw plot[smooth] file {figE1.table};
\draw plot[smooth] file {figE2.table};
\draw plot[smooth] file {figE3.table};
\draw plot[smooth] file {figE4.table};
\smalldot {-1,0} node[anchor=east] {$z_0$};
\smalldot {0.9,0.158} node[anchor=west] {$z_1$};
\smalldot {0.2,0.853} node[anchor=south west] {$z_2$};
\smalldot {0.037,-0.993} node[anchor=north east] {$z_3$};
\end{scope}
}
}

We define a binary operation on the Edwards curve by the hyperbolic
addition law described above.  Let $(-1,0)$, $z_1 = (x_1,y_1)$ and
$z_2=(x_2,y_2)$ be three points on an Edwards curve that are not
collinear (to avoid degenerate cases).  We fit a hyperbola of the
usual form (\ref{eqn:hyp}) through these three points, and let $z_3$
be the fourth point of intersection of the hyperbola with the curve.
We define the hyperbolic sum $z_1\oplus z_2$ of $z_1$ and $z_2$ to be
$\iota(z_3)$.  The following lemma gives an explicit formula for
$z_1\oplus z_2 = \iota(z_3)$.
 
 \begin{lemma}\label{lemma:hyp} 
In this construction, the coordinates are given explicitly by
 \begin{equation}\label{eqn:sum}
 \iota(z_3) = \left(\frac{x_1 x_2 - c y_1 y_2}{1 - d x_1 x_2 y_1 y_2},
\frac{x_1 y_2 + y_1 x_2}{1+d x_1 x_2 y_1 y_2}\right)
 \end{equation}
 \end{lemma}

 This lemma will be proved below (Lemma~\ref{lemma:hyperbola}).  Until
 now, we have assumed the points $(-1,0)$, $z_1$, and $z_2$ are not
 collinear.  Dropping the assumption of non-collinearity, we turn
 formula (\ref{eqn:sum}) of Lemma~\ref{lemma:hyp} into a definition
 and define the hyperbolic sum
\[
 z_1\oplus z_2 := \iota(z_3).
\]
algebraically by that formula in all cases.  We prove below an affine
closure result (Lemma~\ref{lemma:affine}) showing that the
denominators are always nonzero for suitable parameters $c$ and $d$.
In the case of a circle ($c=1$, $d=0$), the formula (\ref{eqn:sum})
reduces to the usual group law (\ref{eqn:cx}).

\section{Group Axioms}\label{sec:axiom}

This section gives an elementary proof of the group axioms for
hyperbolic addition on Edwards curves (Theorem~\ref{thm:group}).  In
this section, we start afresh, shifting away from a geometric language
and work entirely algebraically over an arbitrary field $k$.

\subsection{rings and homomorphisms}

We will assume a basic background in abstract algebra at the level of
a first course (rings, fields, homomorphisms, and kernels).  We set
things up in a way that all of the main identities to be proved are
identities of polynomials with integer coefficients.

If $R$ is a ring (specifically, a ring of polynomials with integer
coefficients), and if $\delta\in R$, then we write $R[\f{\delta}]$ for
the localization of $R$ with respect to the multiplicative set
$S=\{1,\delta,\delta^2,\ldots\}$.  That is, $R[\f{\delta}]$ is the
ring of fractions with numerators in $R$ and denominators in $S$.  We
will need the well-known fact that if $\phi:R\to A$ is a ring
homomorphism that sends $\delta$ to a unit in $A$, then $\phi$ extends
uniquely to a homomorphism $R[\f{\delta}]\to A$ that maps a fraction
$r/\delta^i$ to $\phi(r)\phi(\delta^i)^{-1}$.

\begin{lemma}[kernel property] Suppose that an identity $r = r_1 e_1 +
  r_2 e_2 +\cdots + r_k e_k$ holds in a ring $R$.  If $\phi:R\to A$ is
  a ring homomorphism such that $\phi(e_i) =0$ for all $i$, then
  $\phi(r)=0$.
\end{lemma}

\begin{proof}
$\phi(r) = \sum_{i=1}^k \phi(r_i) \phi(e_i) = 0.$
\end{proof}

We use the following rings: $R_0 := \ring{Z}[c,d]$ and $R_n :=
R_0[x_1,y_1,\ldots,x_n,y_n]$.  We reintroduce the polynomial for the
Edwards curve.  Let
\begin{equation}
e(x,y) = x^2 + c y^2 -1 - d x^2 y^2 \in  R_0[x,y].
\end{equation}

We write $e_i = e(x_i,y_i)$ for the image of the polynomial in $R_j$,
for $i\le j$, under $x\mapsto x_i$ and $y\mapsto y_i$.  Set
$\delta_x = \delta^-$ and $\delta_y = \delta^+$, where
\[\delta^{\pm} (x_1,y_1,x_2,y_2) = 1\pm d x_1 y_1 x_2 y_2\] and
\[
\delta(x_1,y_1,x_2,y_2) = \delta_x\delta_y\in R_2.
\]
We write $\delta_{ij}$ for its image of $\delta$ under
$(x_1,y_1,x_2,y_2)\mapsto (x_i,y_i,x_j,y_j)$.  So,
$\delta=\delta_{12}$.

\subsection{inverse and closure}

We write $z_i = (x_i,y_i)$.
Borrowing the definition of hyperbolic addition from the previous
section, we define a pair of rational functions that we denote using
the symbol $\oplus$:
\begin{equation}\label{eqn:add}
z_1 \oplus z_2 =  \left(\frac{x_1 x_2 - c y_1 y_2}{1 - d x_1 x_2 y_1 y_2},
\frac{x_1 y_2 + y_1 x_2}{1+d x_1 x_2 y_1 y_2}\right) 
\in R_2[\f{\delta}]\times R_2[\f{\delta}].
\end{equation}
Commutativity is a consequence of the subscript symmetry
$1\leftrightarrow 2$ evident in the pair of rational functions:
\[
z_1 \oplus z_2 = z_2\oplus z_1.
\]
If $\phi:R_2[\f{\delta}]\to A$ is a ring homomorphism, we also write
$P_1\oplus P_2\in A^2$ for the image of $z_1\oplus z_2$.  We write
$e(P_i)\in A$ for the image of $e_i=e(z_i)$ under $\phi$.  We often
mark the image $\bar r=\phi(r)$ of an element with a bar accent.

There is an obvious identity element $(1,0)$, expressed as follows.
Under a homomorphism $\phi:R_2[\f{\delta}]\to A$, mapping $z_1\mapsto
P$ and $z_2\mapsto \iota\,P$, we have
\begin{equation}
P\oplus(1,0) = P.
\end{equation}

\begin{lemma} [inverse] 
  Let $\phi:R_2[\f{\delta}]\to A$, with $z_1\mapsto P$, $z_2\mapsto
  \iota(P)$.  If $e(P)=0$, then $P\oplus \iota(P) = (1,0)$.
\end{lemma}

\begin{proof} Plug $P=(a,b)$ and $\iota\,P=(a,-b)$ into
  (\ref{eqn:add}) and simplify using $e(P)=0$.
\end{proof}

\begin{lemma}[closure under addition]\label{lemma:closure}
  Let $\phi:R_2[\f{\delta}]\to A$ with $z_i\mapsto P_i$.  If
  $e(P_1)=e(P_2)=0$, then
  \[
  e(P_1 \oplus P_2) = 0.
  \]
\end{lemma}

\begin{proof} This proof serves as a model for several proofs that are
  based on multivariate polynomial division.  We write
\[
e(z_1\oplus z_2) = \frac{r}{\delta^2},
\]
for some polynomial $r \in R_2$.  It is enough to show that
$\phi(r)=0$.  Polynomial division gives
\begin{equation}\label{eqn:closure}
r= r_1 e_1 + r_2 e_2,
\end{equation}
for some polynomials $r_i\in R_2$.  Concretely, the polynomials $r_i$
are obtained as the output of the one-line Mathematica command
\[
\op{PolynomialReduce}[r,\{e_1,e_2\},\{x_1,x_2,y_1,y_2\}].
\]
The result now follows from the kernel property and
(\ref{eqn:closure}); $ e(P_1) = e(P_2) = 0$ implies $\phi(r)= 0$,
giving ${e}(P_1\oplus P_2)=0$.
\end{proof}

Although the documentation is incomplete, {\tt PolynomialReduce} seems
to be an implementation of a naive multivariate division algorithm
such as \cite{cox1992ideals}.  In particular, our approach does not
require the use of Gr\"obner bases (except in Lemma \ref{lemma:noco}
where they make an easily avoidable appearance).  We write
\[
r \equiv r' \mod S,
\]
where $r-r'$ is a rational function and $S$ is a set of polynomials,
to indicate that the numerator of $r-r'$ has zero remainder when
reduced by polynomial division\footnote{Our computer algebra
  calculations are available at \url{www.github.com/flyspeck}.  This
  includes a formal verification in HOL Light of key polynomial
  identities.} with respect to $S$ using {\tt PolynomialReduce}.  We
also require the denominator of $r-r'$ to be invertible in the
localized polynomial ring.  The zero remainder will give
$\phi(r)=\phi(r')$ in each application.  We extend the notation to
$n$-tuples
\[
(r_1,\ldots,r_n) \equiv (r_1',\ldots,r_n') \mod S,
\]
to mean $r_i \equiv r_i' \mod S$ for each $i$.  Using this approach,
most of the proofs in this article almost write themselves.

\subsection{associativity}

This next step (associativity) is generally considered the hardest
part of the verification of the group law on curves.  Our proof is two
lines and requires little more than polynomial division.  The
polynomials $\delta_x,\delta_y$ appear as denominators in the addition
rule.  The polynomial denominators $\Delta_x,\Delta_y$ that appear
when we add twice are more involved.  Specifically, let $
(x_3',y_3')=(x_1,y_1) \oplus (x_2,y_2)$, let $(x_1',y_1')=(x_2,y_2)
\oplus (x_3,y_3) $, and set
\[
\Delta_x = \delta_x(x_3',y_3',x_3,y_3)
\delta_x(x_1,y_1,x_1',y_1')\delta_{12}\delta_{23}\in R_3.
\]
Define $\Delta_y$ analogously.

\begin{lemma}[generic associativity] \label{lemma:assoc} Let
  $\phi:R_3[\f{\Delta_x\Delta_y}]\to A$ be a homomorphism with
  $z_i\mapsto P_i$.  If $e(P_1)=e(P_2)=e(P_3)=0$, then
\[
(P_1 \oplus P_2)\oplus P_3 = 
P_1 \oplus (P_2\oplus P_3).
\]
\end{lemma}

\begin{proof} By polynomial division in the
  ring $R_3[\f{\Delta_x\Delta_y}]$
\[
((x_1,y_1)\oplus (x_2,y_2)) \oplus (x_3,y_3)\equiv
(x_1,y_1)\oplus ((x_2,y_2) \oplus (x_3,y_3)) \mod \{e_1,e_2,e_3\}.
\]
\end{proof}

\subsection{group law for affine curves}

\begin{lemma}[affine closure] \label{lemma:affine} Let $\phi:R_2\to k$
  be a homomorphism into a field $k$.  If
  $\phi(\delta)=e(P_1)=e(P_2)=0$, then either $\bar d$ or $\bar c \bar
  d$ is a nonzero square in $k$.
\end{lemma}

The lemma is sometimes called completeness, in conflict with the
definition of complete varieties in algebraic geometry.  To avoid
possible confusion, we avoid this terminology.  We use the lemma in
contrapositive form to give conditions on $\bar d$ and $\bar c\bar d$
that imply $\phi(\delta)\ne0$.

\begin{proof} 
  Let $r = (1 - c d y_1^2 y_2 ^2) (1 - d y_1^2 x_2^2)$.  By polynomial
  division,
\begin{equation}\label{eqn:squares}
  r \equiv 0 \mod \{\delta,e_1,e_2\}.
\end{equation}
This forces $\phi(r)=0$, which by the form of $r$ implies that $\bar
c\bar d$ or $\bar d$ is a nonzero square.
\end{proof}

We are ready to state and prove one of the main results of this
article.  This ``ellipstic'' group law is expressed generally enough
to include the group law on the circle and ellipse as a special case
$\bar d = 0$.

\begin{theorem}[group law]\label{thm:group} 
  Let $k$ be a field, let $\bar c \in k$ be a square, and let $\bar
  d\not\in k^{\times 2}$.  
  Then 
  \[
  C= \{P\in k^2 \mid  e(P) = 0\}
  \]
   is an abelian
  group with binary operation $\oplus$.
\end{theorem}

\begin{proof} This follows directly from the earlier results.  For
  example, to check associativity of $P_1\oplus P_2\oplus P_3$, where
  $P_i\in C$, we define a homomorphism $\phi:R_3\to k$ sending
  $z_i\mapsto P_i$ and $(c,d)\mapsto (\bar c,\bar d)$.  By a repeated
  use of the affine closure lemma, $\phi(\Delta_y\Delta_x)$ is nonzero
  and invertible in the field $k$.  The universal property of
  localization extends $\phi$ to a homomorphism
  $\phi:R_3[\f{\Delta_y\Delta_x}]\to k$.  By the associativity lemma
  applied to $\phi$, we obtain the associativity for these three
  (arbitrary) elements of $C$.  The other properties follow similarly
  from the lemmas on closure, inverse, and affine closure.
\end{proof}

The Mathematica calculations in this section are fast. For example,
the associativity certificate takes about $0.12$ second to compute on
a 2.13 GHz processor.  Once the Mathematica code was in final form, it
took less than 30 minutes of development time in HOL Light to copy the
polynomial identities over to the proof assistant and formally verify
them.  All the polynomial identities in this section combined can be
formally verified in less than 2 seconds. The most difficult formal
verification is the associativity identity which takes about 1.5
seconds.

Working with the Weierstrass form of the curve, Friedl was the first
to give a proof of the associative law of elliptic curves in a
computer algebra system (in Cocoa using Gr\"obner bases)
\cite{friedl}.  He writes, ``The verification of some identities took
several hours on a modern computer; this proof could not have been
carried out before the 1980s.''  These identities were eventually
formalized in Coq with runtime 1 minute and 20 seconds
\cite{thery2007proving}.  A non-computational Coq
formalization based on the Picard group appears in
\cite{bartzia2014formal}.  By shifting to Edwards curves, we have
eliminated case splits and
significantly improved the speed of the computational proof.

\section{Group law for projective Edwards curves}

By proving the group laws for a large class of elliptic curves,
Theorem \ref{thm:group} is sufficiently general for many applications
to cryptography.  Nevertheless, to achieve complete generality, we
push forward.

This section show how to remove the restriction $\bar d\not\in
k^{\times 2}$ that appears in the group law in the previous section.
By removing this restriction, we obtain a new proof of the group law
for all elliptic curves in characteristics different from $2$.
Unfortunately, in this section, some case-by-case arguments are
needed, but no hard cases are hidden from the reader.  The level of
exposition here is less elementary than in the previous section.

The basic idea of our construction is that the projective
curve $E$ is obtained by gluing two affine curves $\Eaff$
together.  The associative property for $E$ is a consequence
of the associative property on affine pieces $\Eaff$, which
can be expressed as polynomial identities.

\subsection{definitions}

In this section, we assume that $c\ne 0$ and that $c$ and $d$ are both
squares.  Let $t^2 = d/c$.  By a change of variable $y\mapsto
y/\sqrt{c}$, the Edwards curve takes the form
\begin{equation}\label{eqn:t}
e(x,y)= x^2 + y^2 -1 - t^2 x^2 y^2.
\end{equation}

We assume $t^2\ne 1$.  Note if $t^2=1$, then Equation (\ref{eqn:t})
becomes
\[
-(1-x^2)(1-y^2),
\]
and the curve degenerates to a product of intersecting lines, which
cannot be a group.  We also assume that $t\ne 0$, which excludes the
circle, which has already been fully treated.  Shifting notation for
this new setting, let
\[
R_0 = \ring{Z}[t,\frac{1}{t^2-1},\frac1t],\quad
R_n = R_0[x_1,y_1,\ldots,x_n,y_n].
\]
As before, we write $e_i = e(z_i)$, $z_i=(x_i,y_i)$, and $ e(P_i) =
\phi(e_i)$ when a homomorphism $\phi$ is given.

Define rotation by $\rho(x,y)=(-y,x)$ and inversion $\tau$ by
\[
\tau(x,y) = (1/(tx),1/(ty)).
\]
Let $G$ be the abelian group of order eight generated by $\rho$ and
$\tau$.

\subsection{extended addition}

We extend the binary operation $\oplus$ using the automorphism $\tau$.
We modify notation slightly to write $\oplus_0$ for the binary
operation denoted $\oplus$ until now.  We also write $\delta_1$ for
$\delta$, $\nu_1$ for $\nu$ and so forth.

Set
\begin{equation}\label{eqn:tauplus}
z_1\oplus_1 z_2 := \tau((\tau z_1)\oplus_0 z_2)= \left(\frac{x_1y_1 - x_2 y_2}{x_2
    y_1-x_1 y_2},\frac{x_1 y_1 + x_2 y_2}{x_1 x_2 + y_1 y_2}\right) 
= (\frac{\nu_{1x}}{\delta_{1x}},\frac{\nu_{2x}}{\delta_{2x}})
\end{equation}
in $R_2[\f{\delta_1}]^2$ where $\delta_1 = \delta_{1x}\delta_{1y}$.

We have the following easy identities of rational functions that are
proved by simplification of rational functions:

\noindent{\it inversion invariance}
\begin{align}\label{eqn:r-tau}
\tau (z_1)\oplus_i z_2 &= z_1 \oplus_i \tau z_2;
\end{align}
\noindent{\it rotation invariance}
\begin{align}\label{eqn:r-rho}
\begin{split}
\rho(z_1)\oplus_i z_2 &= \rho(z_1\oplus_i z_2);\\
\delta_i(z_1,\rho z_2) &= \pm \delta_i(z_1,z_2);
\end{split}
\end{align}
\noindent{\it inverse rules  for $\sigma=\tau,\rho$}
\begin{align}\label{eqn:r-iota}
\begin{split}
\iota \sigma(z_1) &= \sigma^{-1} \iota (z_1);\\
\iota (z_1\oplus_i z_2) &= (\iota z_1)\oplus_i (\iota z_2).
\end{split}
\end{align}

The following coherence rule and closure hold by polynomial division:
\begin{align}\label{eqn:r-coh}
\begin{split}
z_1 \oplus_0 z_2 \equiv z_1 \oplus_1 z_2 &\mod \{e_1,e_2\};\\
e(z_1\oplus_1 z_2) \equiv 0 &\mod \{e_1,e_2\}.
\end{split}
\end{align}
The first identity requires inverting $\delta_0\delta_1$ and the
second requires inverting $\delta_1$.

\subsection{projective curve and dichotomy}

Let $k$ be a field of characteristic different from two.  We let
$\Eaff$ be the set of zeros of Equation (\ref{eqn:t}) in $k^2$.  Let
$\Eoo\subset \Eaff$ be the subset of $\Eaff$ with nonzero coordinates
$x,y\ne0$.

We construct the projective Edwards curve $E$ by taking two copies of
$\Eaff$, glued along $\Eoo$ by isomorphism $\tau$.  We write $[P,i]\in
E$, with $i\in \ring{Z}/2\ring{Z}=\ring{F}_2$, for the image of $P\in
\Eaff$ in $E$ using the $i$th copy of $\Eaff$.  The gluing condition
gives for $P\in \Eoo$:
\begin{equation}\label{eqn:glue}
[P,i]=[\tau P,i+1].
\end{equation}

The group $G$ acts on the set $E$, specified on generators $\rho,\tau$
by $\rho[P,i]=[\rho(P),i]$ and $\tau[P,i]=[P,i+1]$.

We define addition on $E$ by
\begin{equation}\label{eqn:add-proj}
[P,i]\oplus [Q,j] = [P\oplus_\ell Q,i+j],\quad 
\text{if } \delta_\ell(P,Q)\ne 0,\quad \ell\in\ring{F}_2
\end{equation}
We will show that the addition is well-defined, is defined for all
pairs of points in $E$, and that it gives a group law with identity
element $[(1,0),0]$.  The inverse is $[P,i]\mapsto [\iota P,i]$, which
is well-defined by the inverse rules (\ref{eqn:r-iota}).

\begin{lemma} \label{lemma:no-fix} $G$ acts without fixed point on
  $\Eoo$.  That is, $g P = P$ implies that $g=1_G\in G$.
\end{lemma}

\begin{proof} Write $P=(x,y)$.  If $g = \rho^k\ne 1_G$, then $g P = P$
  implies that $2x=2y=0$ and and $x=y=0$ (if the characteristic is not
  two), which is not a point on the curve.  If $g = \tau \rho^k$, then
  the fixed-point condition $g P = P$ leads to $2t x y=0$ or $t x^2 =
  t y^2 =\pm 1$.  Then $e(x,y) = 2 (\pm1-t)/t\ne0$, and again $P$ is
  not a point on the curve.
\end{proof}

The domain of $\oplus_i$ is
\[
\Eaf{i} := \{(P,Q)\in \Eaff^2\mid \delta_i(P,Q)\ne0\}.
\]
Whenever we write $P\oplus_i Q$, it is always accompanied by the
implicit assertion that $(P,Q)\in \Eaf{i}$.

There is a group isomorphism $\ang{\rho}\to \Eaff\setminus\Eoo$ given by
\[
g\mapsto g(1,0)\in\{\pm (1,0),\pm (0,1)\} = \Eaff\setminus \Eoo.
\]

\begin{lemma}[dichotomy]\label{lemma:noco} 
\par
\noindent  
Let $P,Q\in \Eaff$.  Then either $P\in \Eoo$ and $Q=g \iota\, P$ for
some $g\in \tau\ang{\rho}$, or $(P,Q)\in \Eaf{i}$ for some $i$.
Moreover, assume that $P\oplus_i Q = (1,0)$ for some $i$, then $Q =
\iota\,P$.
\end{lemma}

\begin{proof}  We start with the first claim.
We analyze the denominators in the formulas for $\oplus_i$.  
We have $(P,Q)\in\Eaf{0}$ for all $P$ or $Q\in \Eaff\setminus\Eoo$.
That case completed,  we may assume that $P,Q\in \Eoo$.
  Assuming
  \[
  \delta_0(P,Q) = \delta_{0x}(P,Q)\delta_{0y}(P,Q)=0,\quad\text{and}\quad
  \delta_1(P,Q) = \delta_{1x}(P,Q)\delta_{1y}(P,Q)=0,
  \]
  we show that $Q = g \iota P$ for some $g\in \tau\ang{\rho}$.
  Replacing $Q$ by $\rho Q$ if needed, which exchanges
  $\delta_{0x}\leftrightarrow \delta_{0y}$, we may assume that
  $\delta_{0x}(P,Q)=0$.  Set $\tau Q = Q_0 = (a_0,b_0)$ and
  $P=(a_1,b_1)$.  

We claim that
\begin{equation}\label{eqn:gp-}
(a_0,b_0) \in \{\pm (b_1,a_1)\} \subset \Go\iota\,P.
\end{equation}
Write $\delta',\delta_{+},\delta_{-}$ for $x_0 y_0\delta_{0x}$, $t x_0
y_0\delta_{1x}$, and $t x_0 y_0 \delta_{1y}$ respectively, each
evaluated at $(P,\tau(Q_0))=(x_1,y_1,1/(t x_0),1/(t y_0))$.  (The
nonzero factors $x_0y_0$ and $t x_0 y_0$ have been included to clear
denominators, leaving us with polynomials.)

We have two cases $\pm$, according to $\delta_{\pm}=0$.  In each case,
let
\[
S_\pm = \text{Gr\"obner basis of } \{e_1,e_2, 
\delta',\delta_{\pm},q x_0 x_1 y_0 y_1 - 1\}.
\]
The polynomial $q x_0 x_1 y_0 y_1-1$ is included to encode the
condition $a_0,b_0,a_1,b_1\ne 0$, which holds on $\Eoo$.  Polynomial
division gives
\begin{equation}\label{eqn:dichot}
(x_0^2-x_1^2,y_0^2-x_1^2,x_0 y_0 - x_1 y_1) \equiv (0,0,0) \mod S_\pm.
\end{equation}
These equations immediately yield $(a_0,b_0) = \pm (b_1,a_1)$ and
(\ref{eqn:gp-}).  This gives the claim.  In summary, we have $\tau Q =
Q_0 = (a_0,b_0) = g \iota\,P$, for some $g\in \Go$.  Then $Q = \tau g
\iota\,P$.

The second statement of the lemma has a similar proof.  Polynomial
division gives for $i\in \ring{F}_2$:
\[
z_1 \equiv \iota (z_2) \mod \text{Gr\"obner} 
\{ e_1,e_2,q x_1 y_1 x_2 y_2 -1,\nu_{i y},\nu_{i x}-\delta_{i x} \}.
\]
Note that $\nu_{i y}=\nu_{i x}-\delta_{i x}=0$ is the condition for
the sum to be the identity element:
$(1,0)=(\nu_{ix}/\delta_{ix},\nu_{iy}/\delta_{iy})$.
\end{proof}

\begin{lemma}[covering] The rule (\ref{eqn:add-proj}) defining
  $\oplus$ assigns at least one value for every pair of points in $E$.
\end{lemma}

\begin{proof} If $Q=\tau \rho^k \iota\,P$, then $\tau Q$ does not have
  the form $\tau\rho^k\iota P$ because the action of $G$ is
  fixed-point free.  By dichotomy,
\begin{equation}\label{eqn:tt}
[P,i]\oplus [Q,j] = [P\oplus_\ell \tau Q,i+j+1]
\end{equation}
works for some $\ell$.  Otherwise, by dichotomy $P\oplus_\ell Q$ is
defined for some $\ell$.
\end{proof}

\begin{lemma}[well-defined] Addition $\oplus$ given by
  (\ref{eqn:add-proj}) on $E$ is well-defined.
\end{lemma}

\begin{proof}
  The right-hand side of (\ref{eqn:add-proj}) is well-defined by
  coherence (\ref{eqn:r-coh}), provided we show well-definedness
  across gluings (\ref{eqn:glue}).  We use dichotomy.  If $Q=\tau
  \rho^k \iota\,P$, then by an easy simplification of polynomials,
\[
\delta_0(z,\tau\rho^k\iota z)=\delta_1(z,\tau\rho^k\iota z)=0.
\]
so that only one rule (\ref{eqn:tt}) for $\oplus$ applies (up to
coherence (\ref{eqn:r-coh}) and inversion (\ref{eqn:r-tau})), making
it necessarily well-defined.  Otherwise, coherence (\ref{eqn:r-coh}),
inversion (\ref{eqn:r-tau}), and (\ref{eqn:tauplus})) give when
$[Q,j]=[\tau Q,j+1]$:
 \[ 
[P\oplus_k \tau Q,i+j+1]=[\tau(P\oplus_k \tau Q),i+j] =
 [P\oplus_{k+1} Q,i+j] = [P\oplus_\ell Q,i+j].
\]
\end{proof}

\subsection{group}

\begin{theorem}  $E$ is an abelian group.
\end{theorem}

\begin{proof} We have already shown the existence of an identity and
  inverse.

  We prove associativity.  Both sides of the associativity identity
  are clearly invariant under shifts $[P,i]\mapsto [P,i+j]$ of the
  indices.  Thus, it is enough to show
\[
[P,0] \oplus ([Q,0]\oplus [R,0]) = ([P,0]\oplus [Q,0])\oplus [R,0].
\]
By polynomial division, we have the following associativity identities
\begin{equation}\label{eqn:assoc-affine}
 (z_1\oplus_k z_2)\oplus_\ell z_3 \equiv z_1 
\oplus_i (z_2\oplus_j z_3) \mod \{e_1,e_2,e_3\}
\end{equation}
in the appropriate localizations, for $i,j,k,\ell\in \ring{F}_2$.

Note that $(g [P_1,i])\oplus [P_2,j] = g([P_1,i]\oplus [P_2,j])$ for
$g\in G$, as can easily be checked on generators $g=\tau,\rho$ of $G$,
using dichotomy, (\ref{eqn:add-proj}), and (\ref{eqn:r-rho}).  We use
this to cancel group elements $g$ from both sides of equations without
further comment.

We claim that
\begin{equation}\label{eqn:semi}
([P,0]\oplus [Q,0])\oplus [\iota\,Q,0] = [P,0].
\end{equation}
The special case $Q= \tau\rho^k \iota(P)$ is easy. We reduce the claim
to the case where $P\oplus_\ell Q\ne \tau\rho^k Q$, by applying $\tau$
to both sides of (\ref{eqn:semi}) and replacing $P$ with $\tau P$ if
necessary.  Then by dichotomy, the left-hand side simplifies by affine
associativity \ref{eqn:assoc-affine} to give the claim.

Finally, we have general associativity by repeated use of dichotomy,
which reduces in each case to (\ref{eqn:assoc-affine}) or
(\ref{eqn:semi}).
\end{proof}

When the characteristic of $k$ is two, we have
\[
e(x,y)= x^2 + y^2 - 1 - t^2 x^2 y^2 
= ( x y + p(x + 1) + q y)^2 t^2,\quad p=q=t^{-1},
\]
so that the Edwards curve is itself a hyperbola in our family and the
group law is invalid.  The sum $(x,x)\oplus_i (x,x)$ is not defined
when $x^2 t=1$.

\section{Hyperbola revisited}

Our proof of the group axioms in the previous sections does not
logically depend on the geometric interpretation of addition as
intersection points with a hyperbola.  Here we show that the addition
formula (\ref{eqn:sum}) is indeed given by hyperbolic addition (when a
determinant $D$ is nonzero).  We revert to the meaning of $R_0$ and
$R_n$ from Section \ref{sec:axiom}.

\subsection{addition}

Three points $(x_0,y_0)$, $(x_1,y_1)$, and $(x_2,y_2)$ in the plane
are collinear if and only if the following determinant is zero:
\[
\begin{vmatrix}
x_0 & y_0 & 1\\
x_1 & y_1 & 1\\
x_2 & y_2 & 1
\end{vmatrix}.
\]
When $(x_0,y_0) = (-1,0)$, the determinant
is $D= (x_1+1) y_2 - (x_2+1) y_1\in R_2$. 
We recall the polynomial
\[
h(p,q,x,y) = x y + p (x+1) + q y \in R_0[p,q,x,y]
\]
representing a family of hyperbolas.
We can solve the two linear equations
\begin{equation}\label{eqn:pq}
h(p,q,x_1,y_1)=h(p,q,x_2,y_2)=0
\end{equation}
uniquely for $p=p_0$ and $q=q_0$ in the ring $R_2[\f{D}]$ to obtain
$h(x,y) = h(p_0,q_0,x,y) \in R_2[\f{D}][x,y]$.  It represents the
unique hyperbola in the family passing through points and $(-1,0)$,
$(x_1,y_1)$, and $(x_2,y_2)$.
 
\begin{lemma}[hyperbolic addition]\label{lemma:hyperbola}
  Let $\phi:R_2[\f{D\delta}]\to A$ be a ring homomorphism 
  If $e(P_1) =  e(P_2) = 0$, then
  $\bar h(\iota(P_1\oplus P_2))=0$.
\end{lemma}

\begin{proof}
  We work in the ring $R_2[\f{D\delta}]$ and write
\[
h(x_3',-y_3') = \frac{r}{D\delta},\quad 
\text{where } (x_3',y_3') = (x_1,y_1)\oplus (x_2,y_2)
\]
for some polynomial $r \in R_2$.  Polynomial division gives
\begin{equation}\label{eqn:h}
r \equiv 0 \mod \{e_1,e_2\}.
\end{equation}
\end{proof}

\subsection{group law based on divisors}

In this subsection, we sketch a second proof of the group law for
Edwards curves that imitates a standard proof for the chord and
tangent construction for Weierstrass curves.  The proof is not as
elementary as our first proof, but it achieves greater conceptual
simplicity.  Here we work over an algebraically closed field $k$ of
characteristic different from $2$.

\begin{theorem} Let $E$ be an Edwards curve of genus one with
  $k$-points $E(k)\subset \ring{P}_k^1\times \ring{P}_k^1$.  Then
  $E(k)$ is a group under hyperbolic addition.
\end{theorem}

\begin{proof}[Proof sketch] 
  Each hyperbola in the family (\ref{eqn:hyp}) determines a rational
  function in the function field of $E$:
\[
\frac{x y + p (x+1) + q y}{x y}.
\]
Its divisor has the form
\begin{equation}\label{eqn:h}
[P] + [Q] + [R] - [(1,0)] - [(0,1)] - [(0,-1)],
\end{equation}
for some points $P,Q,R\in E(k)$.  (In particular, no hidden zeros or
poles lurk at infinity.)  By the definition of hyperbolic addition,
$\iota(R) = P\oplus Q$.  Conversely, three points on the Edwards curve
that sum to $0$ determine a rational function in the family (by
solving Equation (\ref{eqn:pq}) for $p$ and $q$ using any two of the
three points). We consider six rational function
$f_1,f_2,f_3,f_1',f_2',f_3'$ constructed in this manner, where each
rational function is specified by three points of $E(k)$ as indicated
in Figure {\ref{fig:9pt}}.  Each line segment in the figure represents
a hyperbola in our family.

We compute the divisor using (\ref{eqn:h})
\[
\op{div} (f) =\op{div}\left(\frac{f_1 f_2 f_3}{f_1' f_2' f_3'}\right) = 
[P\oplus (Q\oplus R)] - [(P\oplus Q) \oplus R].
\]
Other poles and zeros occur twice with opposite sign.
If the two points on the right are distinct, the function $f$ has exactly one
simple zero and one solitary simple pole.  Such a function does not
exist on a curve of positive genus, so $f$ is constant, $\op{div}(f)=0$, and
\[
P\oplus (Q\oplus R) = (P\oplus Q) \oplus R.
\]
\end{proof}

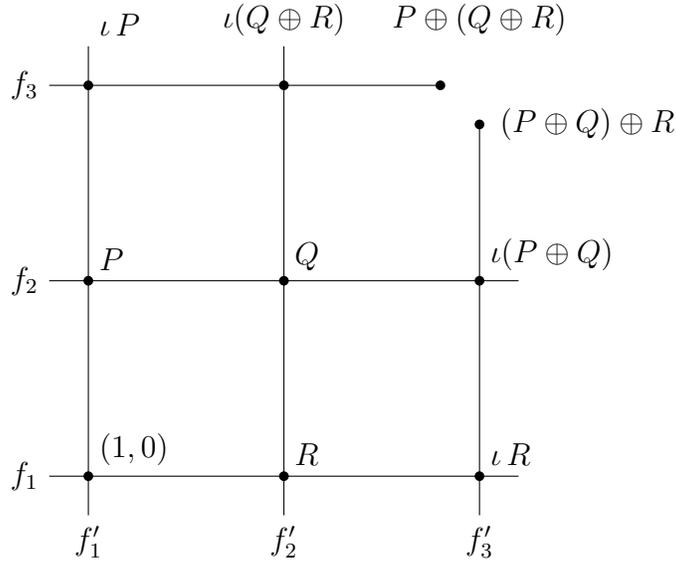
\tikzfig{9pt}{We define six hyperbolas in our family (each represented
  here as a line segment), specifying each by their three non-fixed
  points on $E$.  The three given points on each hyperbola sum to
  zero.}{
{
\begin{scope}[xshift=0,scale=1.3]
\def\t{0.4}
\draw (0,0-\t) -- (0,2) -- (0,4+\t);
\draw (2,0-\t) -- (2,2) -- (2,4+\t);
\draw (4,0-\t) -- (4,2) -- (4,4-\t);
\draw (0-\t,0) -- (2,0) -- (4+\t,0);
\draw (0-\t,2) -- (2,2) -- (4+\t,2);
\draw (0-\t,4) -- (2,4) -- (4-\t,4);
\draw (0,0) node[anchor=south west] {$(1,0)$};
\draw (2,0) node[anchor=south west] {$R$};
\draw (4,0) node[anchor=south west] {$\iota\, R$};
\draw (0,2) node[anchor=south west] {$P$};
\draw (2,2) node[anchor=south west] {$Q$};
\draw (4,2) node[anchor=south west] {$\iota(P\oplus Q)$};
\draw (0,4+\t) node[anchor=south west] {$\iota\,P$};
\draw (2,4+\t) node[anchor=south] {$\iota(Q\oplus R)$};
\draw (4,4+\t) node[anchor=south] {$P\oplus (Q\oplus R)$};
\draw (4+0.1,4-\t) node[anchor=west] {$(P \oplus Q)\oplus R$};
\draw (-\t,0) node[anchor=east] {$f_1$};
\draw (-\t,2) node[anchor=east] {$f_2$};
\draw (-\t,4) node[anchor=east] {$f_3$};
\draw (0,-\t) node[anchor=north] {$f_1'$};
\draw (2,-\t) node[anchor=north] {$f_2'$};
\draw (4,-\t) node[anchor=north] {$f_3'$};
\smalldot{0,0};
\smalldot{0,2};
\smalldot{0,4};
\smalldot{2,0};
\smalldot{2,2};
\smalldot{2,4};
\smalldot{4,0};
\smalldot{4,2};
\smalldot{4,4-\t};
\smalldot{4-\t,4};
\end{scope}
}
}

Unlike our first proof, this proof is not self-contained because we
rely on the fact that on a curve of genus one, no function has a
single simple pole.

\subsection{elliptic curves}\label{sec:elliptic}

We retain the assumption that the characteristic of the field is not
$2$.

Starting with the equation $x^2 (1-d y^2) = (1 - c y^2)$ of the
Edwards curve, we can multiply both sides by $(1- d y^2)$ to bring it
into the form
\[
w^2 = (1 - d y^2) (1 - c y^2),
\]
where $w = x (1 - d y^2)$.  This a Jacobi quartic.  It is an elliptic
curve whenever the polynomial in $y$ on the right-hand side has degree
four and is separable.  In particular, if $c=1$ and $d=t^2\ne 1$, it
is an elliptic curve.

After passing to a quadratic extension if necessary, every elliptic
curve is isomorphic to an Edwards curve.  This observation can be used
to give a new proof that a general elliptic curve $E$ (say chord and
tangent addition in Weierstrass form) is a group.  To carry this out,
write the explicit isomorphism $E\to E'$ to an Edwards curve taking
the binary operation $\oplus_E$ on $E$ to the group operation on the
Edwards curve.  Then associativity for $\oplus_E$ follows from
associativity on the Edwards curve.

\section{Acknowledgements}

A number of calculations are reworkings of calculations found in
Edwards, Bernstein, Lange et al.~\cite{edwards2007normal},
\cite{bernstein2008twisted}, \cite{bernstein2007faster}.  Inspiration
for this article comes from Bernstein and Lange's wonderfully gentle
introduction to elliptic-curve cryptography at the 31st Chaos
Communication Congress in December, 2014.  They use the group law on
the circle to motivate the group law on Edwards elliptic curves.
Hyperbolic addition is introduced for Edwards elliptic curves in
\cite{arene2011faster}.  

\bibliography{refs} 
\bibliographystyle{alpha}

\end{document}